\documentclass[12pt]{amsart}
\date{\today}

\input xymatrix
\xyoption{all}

\usepackage{amsmath,amsthm,amsfonts,amssymb,mathrsfs}

\usepackage{color}

\usepackage{hyperref}

\newtheorem{theorem}{Theorem}[section]
\newtheorem{proposition}[theorem]{Proposition}
\newtheorem{corollary}[theorem]{Corollary}
\newtheorem{lemma}[theorem]{Lemma}

\theoremstyle{definition}
\newtheorem{example}[theorem]{Example}
\newtheorem{remark}[theorem]{Remark}
\newtheorem{definition}[theorem]{Definition}

\begin{document}

\title[On inverse submonoids of the monoid ...]{On inverse submonoids of the monoid of almost monotone injective co-finite partial selfmaps of
positive integers}

\author[O.~Gutik and A.~Savchuk]{Oleg~Gutik and Anatolii~Savchuk}
\address{Faculty of Mathematics, Ivan Franko National
University of Lviv, Universytetska 1, Lviv, 79000, Ukraine}
\email{ogutik@gmail.com, ovgutik@yahoo.com, asavchuk1@meta.ua, asavchuk3333@gmail.com}

\keywords{Semigroup of isometries, inverse semigroup, partial bijection, congruence, bicyclic semigroup,  semitopological semigroup, topological semigroup, discrete topology, embedding, Bohr compactification}

\subjclass[2010]{20M18, 20M20, 20M30, 22A15,  54D40, 54D45, 54H10}

\begin{abstract}
In this paper we study submonoids of the monoid $\mathscr{I}_{\infty}^{\,\Rsh\!\!\!\nearrow}(\mathbb{N})$ of almost monotone injective co-finite partial selfmaps of positive integers $\mathbb{N}$.
Let $\mathscr{I}_{\infty}^{\!\nearrow}(\mathbb{N})$ be a submonoid of $\mathscr{I}_{\infty}^{\,\Rsh\!\!\!\nearrow}(\mathbb{N})$ which consists of cofinite monotone partial bijections of $\mathbb{N}$ and $\mathscr{C}_{\mathbb{N}}$ be a subsemigroup of $\mathscr{I}_{\infty}^{\,\Rsh\!\!\!\nearrow}(\mathbb{N})$ which is generated by the partial shift $n\mapsto n+1$ and its inverse partial map.
We show that every automorphism of a full inverse subsemigroup of $\mathscr{I}_{\infty}^{\!\nearrow}(\mathbb{N})$ which contains the semigroup $\mathscr{C}_{\mathbb{N}}$ is the identity map. We construct a submonoid $\mathbf{I}\mathbb{N}_{\infty}^{[\underline{1}]}$ of $\mathscr{I}_{\infty}^{\,\Rsh\!\!\!\nearrow}(\mathbb{N})$ with the following property: if $S$ is an inverse submonoid of $\mathscr{I}_{\infty}^{\,\Rsh\!\!\!\nearrow}(\mathbb{N})$ such that $S$ contains $\mathbf{I}\mathbb{N}_{\infty}^{[\underline{1}]}$ as a submonoid, then every non-identity congruence $\mathfrak{C}$ on $S$ is a group congruence.
We show that if $S$ is an inverse submonoid of $\mathscr{I}_{\infty}^{\,\Rsh\!\!\!\nearrow}(\mathbb{N})$ such that $S$ contains $\mathscr{C}_{\mathbb{N}}$ as a submonoid then $S$ is simple and the quotient semigroup $S/\mathfrak{C}_{\mathbf{mg}}$, where $\mathfrak{C}_{\mathbf{mg}}$ is minimum group congruence on $S$, is isomorphic to the additive group of integers. Also, we study topologizations of inverse submonoids of $\mathscr{I}_{\infty}^{\,\Rsh\!\!\!\nearrow}(\mathbb{N})$ which contain $\mathscr{C}_{\mathbb{N}}$ and embeddings of  such semigroups into compact-like topological semigroups.
\end{abstract}

\maketitle


\section{Introduction and preliminaries}
In this paper all spaces will be
assumed to be Hausdorff. Furthermore
we shall follow the terminology of \cite{Carruth-Hildebrant-Koch-1983-1986, Clifford-Preston-1961-1967, Engelking-1989, Lawson-1998, Ruppert-1984}.
 We shall denote the set of all positive integers by $\mathbb{N}$, the first infinite ordinal by $\omega$
 and  the cardinality of the set $A$
 by
$|A|$. If $A$ is a subset of a semigroup $S$ then by $\langle A\rangle$  we shall
denote a subsemigroup of $S$ generated by the elements of the set $A$.

An algebraic semigroup $S$ is called {\it inverse} if for any
element $x\in S$ there exists a unique $x^{-1}\in S$ such that
$xx^{-1}x=x$ and $x^{-1}xx^{-1}=x^{-1}$. The element $x^{-1}$ is
called the {\it inverse of} $x\in S$. If $S$ is an inverse
semigroup, then the function $\operatorname{inv}\colon S\to S$
which assigns to every element $x$ of $S$ its inverse element
$x^{-1}$ is called an {\it inversion}.

A congruence $\mathfrak{C}$ on a semigroup $S$ is called
\emph{non-trivial} if $\mathfrak{C}$ is distinct from universal and
identity congruences on $S$, and a \emph{group congruence} if the quotient
semigroup $S/\mathfrak{C}$ is a group. If $\mathfrak{C}$ is a congruence on a semigroup $S$ then by $\mathfrak{C}^{\sharp}$ we denote the natural homomorphism from $S$ onto the quotient semigroup $S/\mathfrak{C}$.

If $S$ is a semigroup, then we shall denote the subset of all
idempotents in $S$ by $E(S)$. If $S$ is an inverse semigroup, then
$E(S)$ is closed under multiplication and we shall refer to $E(S)$ a as
\emph{band} (or the \emph{band of} $S$). Then the semigroup
operation on $S$ determines the following partial order $\preccurlyeq$
on $E(S)$: $e\preccurlyeq f$ if and only if $ef=fe=e$. This order is
called the {\em natural partial order} on $E(S)$. A
\emph{semilattice} is a commutative semigroup of idempotents.

An inverse subsemigroup $T$ of an inverse semigroup $S$ is called \emph{full} if $E(S)=E(T)$.

By
$(\mathscr{P}_{<\omega}(\lambda),\cup)$ we shall denote the
\emph{free semilattice with identity} over a set of cardinality
$\lambda\geqslant\omega$, i.e.,
$(\mathscr{P}_{<\omega}(\lambda),\cup)$ is the set of all finite
subsets (with the empty set) of $\lambda$ with the semilattice
operation ``union''.

If $S$ is a semigroup, then we shall denote the Green relations on $S$ by $\mathscr{R}$, $\mathscr{L}$, $\mathscr{J}$, $\mathscr{D}$ and $\mathscr{H}$ (see \cite{Clifford-Preston-1961-1967}). A semigroup $S$ is called \emph{simple} if $S$ does not contain proper two-sided ideals and \emph{bisimple} if $S$ has only one $\mathscr{D}$-class.

A (\emph{semi})\emph{topological} \emph{semigroup} is a topological space with a (separately) continuous semigroup operation. An inverse topological semigroup with continuous inversion is called a \emph{topological inverse semigroup}.

A topology $\tau$ on a semigroup $S$ is called:
\begin{itemize}
  \item \emph{semigroup} if $(S,\tau)$ is a topological semigroup;
  \item \emph{semigroup inverse} if $S$ is an inverse semigroup and $(S,\tau)$ is a topological inverse semigroup;
  \item \emph{shift-continuous} if $(S,\tau)$ is a semitopological semigroup.
\end{itemize}

The \emph{bicyclic semigroup} (or the \emph{bicyclic monoid}) $\mathscr{C}(p,q)$ is the semigroup with the identity $1$ generated by two elements $p$ and $q$, subject only to the condition $pq=1$.

The bicyclic semigroup is bisimple and every one of its congruences is either trivial or a group congruence. Moreover, every homomorphism $h$ of the bicyclic semigroup is either an isomorphism or the image of ${\mathscr{C}}(p,q)$ under $h$ is a cyclic group~(see \cite[Corollary~1.32]{Clifford-Preston-1961-1967}). The bicyclic semigroup plays an important role in algebraic theory of semigroups and in the theory of topological semigroups. For example a well-known Andersen's result~\cite{Andersen-1952} states that a ($0$--)simple semigroup with an idempotent is completely ($0$--)simple if and only if it does not contain an isomorphic copy of the bicyclic semigroup. The bicyclic monoid admits only the discrete semigroup Hausdorff topology. Bertman and  West in \cite{Bertman-West-1976} extended this result for the case of Hausdorff semitopological semigroups. Stable and $\Gamma$-compact topological semigroups do not contain the bicyclic monoid~\cite{Anderson-Hunter-Koch-1965, Hildebrant-Koch-1986}. The problem of embedding of the bicyclic monoid into compact-like topological semigroups was studied in \cite{Banakh-Dimitrova-Gutik-2009, Banakh-Dimitrova-Gutik-2010, Gutik-Repovs-2007}.
Independently to Eberhart-Selden results on topolozabilty of the bicyclic semigroup, in \cite{Taimanov-1973} Taimanov constructed a commutative semigroup $\mathfrak{A}_\kappa$ of cardinality $\kappa$ which admits only the discrete semigroup topology. Also, Taimanov \cite{Taimanov-1975} gave sufficient conditions for a commutative semigroup to have a non-discrete semigroup topology. In the paper \cite{Gutik-2016} it was showed that for the Taimanov semigroup $\mathfrak{A}_\kappa$ from \cite{Taimanov-1973} the following conditions hold:
every $T_1$-topology $\tau$ on the semigroup $\mathfrak{A}_\kappa$ such that $(\mathfrak{A}_\kappa,\tau)$  is a topological semigroup is discrete; $\mathfrak A_\kappa$ is closed in any $T_1$-topological semigroup containing $\mathfrak A_\kappa$
and every homomorphic non-isomorphic image of $\mathfrak{A}_\kappa$ is a zero-semigroup.

Non-discrete topologizations of some bicyclic-like semigroups were studied in \cite{Bardyla-2016, Bardyla-2016a, Bardyla-2018, Bardyla-2018a, Bardyla-20??, Bardyla-Gutik-2016, Gutik-2015, Gutik-Maksymyk-2016, Hogan-1984, Mesyan-Mitchell-Morayne-Peresse-2016, Selden_A-1985}.
In particular in \cite{Fihel-Gutik-2011} it is proved that the discrete topology is the unique shift-continuous Hausdorff topology on the extended bicyclic semigroup $\mathscr{C}_{\mathbb{Z}}$. We observe that for many ($0$-)bisimple semigroups   $S$ the following statement holds: \emph{every shift-continuous Hausdorff Baire (in particular locally compact) topology on $S$ is discrete} (see \cite{Chuchman-Gutik-2010, Gutik-2018, Gutik-Maksymyk-2016a, Gutik-Pozdnyakova-2014, Gutik-Repovs-2011, Gutik-Repovs-2012}).

Let $\mathscr{I}_\lambda$ denote the set of all partial one-to-one
transformations of a set $X$ of cardinality $\lambda$ together
with the following semigroup operation:
\begin{equation*}
x(\alpha\beta)=(x\alpha)\beta \quad \hbox{if} \quad x\in\operatorname{dom}(\alpha\beta)=\{
y\in\operatorname{dom}\alpha\mid y\alpha\in\operatorname{dom}\beta\}, \qquad  \hbox{for} \quad
\alpha,\beta\in\mathscr{I}_\lambda.
\end{equation*}
 The semigroup
$\mathscr{I}_\lambda$ is called the \emph{symmetric inverse
semigroup} over the set $X$~(see \cite{Clifford-Preston-1961-1967}). The symmetric inverse
semigroup was introduced by Wagner~\cite{Wagner-1952} and it plays
a major role in the theory of semigroups.

\begin{remark}\label{remark-1.1}
We observe that the bicyclic semigroup is isomorphic to the
semigroup $\mathscr{C}_{\mathbb{N}}$ which is
generated by partial transformations $\alpha$ and $\beta$ of the set
of positive integers $\mathbb{N}$, defined as follows:
\begin{equation*}
\operatorname{dom}\alpha=\mathbb{N}, \qquad \operatorname{ran}\alpha=\mathbb{N}\setminus\{1\},  \qquad (n)\alpha=n+1
\end{equation*}
and
\begin{equation*}
\operatorname{dom}\beta=\mathbb{N}\setminus\{1\}, \qquad \operatorname{ran}\beta=\mathbb{N},  \qquad (n)\beta=n-1
\end{equation*}
(see Exercise~IV.1.11$(ii)$ in \cite{Petrich-1984}).
\end{remark}

Let $\mathbb{N}$ be the set of all positive integers. We shall denote
the semigroup of monotone, non-decreasing, injective partial
transformations $\varphi$ of $\mathbb{N}$ such that the sets
$\mathbb{N}\setminus\operatorname{dom}\varphi$ and
$\mathbb{N}\setminus\operatorname{rank}\varphi$ are finite by
$\mathscr{I}_{\infty}^{\!\nearrow}(\mathbb{N})$.
Obviously, $\mathscr{I}_{\infty}^{\!\nearrow}(\mathbb{N})$ is an
inverse subsemigroup of the semigroup $\mathscr{I}_\omega$. The
semigroup $\mathscr{I}_{\infty}^{\!\nearrow}(\mathbb{N})$ is
called \emph{the semigroup of cofinite monotone partial
bijections} of $\mathbb{N}$.

In \cite{Gutik-Repovs-2011} Gutik and Repov\v{s} studied the
semigroup $\mathscr{I}_{\infty}^{\!\nearrow}(\mathbb{N})$. They showed that the semigroup
$\mathscr{I}_{\infty}^{\!\nearrow}(\mathbb{N})$ has algebraic
properties similar to the bicyclic semigroup: it is bisimple and all
of its non-trivial group homomorphisms are either isomorphisms or
group homomorphisms. Also, they proved that every locally compact inverse semigroup topology
$\tau$ on $\mathscr{I}_{\infty}^{\!\nearrow}(\mathbb{N})$ is discrete and described the closure
of $(\mathscr{I}_{\infty}^{\!\nearrow}(\mathbb{N}),\tau)$ in a
topological semigroup.

Doroshenko in \cite{Doroshenko2005, Doroshenko2009} studied the semigroups of endomorphisms of linearly ordered sets $\mathbb{N}$ and $\mathbb{Z}$ and their subsemigroups of cofinite endomorphisms $\mathcal{O}_{fin}(\mathbb{N})$ and $\mathcal{O}_{fin}(\mathbb{Z})$. In \cite{Doroshenko2009} he described the Green relations, groups of automorphisms, conjugacy, centralizers of elements, growth, and free subsemigroups in these subgroups. Especially in \cite{Doroshenko2009} it is proved that the  group  of  automorphisms  consists  only  of  the  identity  mapping,  whereas  the  groups  of  automorphisms of $\mathcal{O}_{fin}(\mathbb{Z})$ is isomorphic to the semigroup of integers with operation of addition and consist only of inner automorphisms.
In \cite{Doroshenko2005} there was shown  that  both these semigroups do not admit an irreducible system of generators. In their  subsemigroups of cofinite functions all irreducible systems of generators are
described there. Also, here the last semigroups are presented in terms of generators and relations.

A partial map $\alpha\colon \mathbb{N}\rightharpoonup \mathbb{N}$ is
called \emph{almost monotone} if there exists a finite subset $A$ of
$\mathbb{N}$ such that the restriction
$\alpha\mid_{\mathbb{N}\setminus A}\colon \mathbb{N}\setminus
A\rightharpoonup \mathbb{N}$ is a monotone partial map.

By $\mathscr{I}_{\infty}^{\,\Rsh\!\!\!\nearrow}(\mathbb{N})$ we
shall denote the semigroup of monotone, almost non-decreasing,
injective partial transformations of $\mathbb{N}$ such that the sets
$\mathbb{N}\setminus\operatorname{dom}\varphi$ and
$\mathbb{N}\setminus\operatorname{rank}\varphi$ are finite for all
$\varphi\in\mathscr{I}_{\infty}^{\,\Rsh\!\!\!\nearrow}(\mathbb{N})$.
Obviously, $\mathscr{I}_{\infty}^{\,\Rsh\!\!\!\nearrow}(\mathbb{N})$
is an inverse subsemigroup of the semigroup $\mathscr{I}_\omega$ and
the semigroup $\mathscr{I}_{\infty}^{\!\nearrow}(\mathbb{N})$ is an
inverse subsemigroup of
$\mathscr{I}_{\infty}^{\,\Rsh\!\!\!\nearrow}(\mathbb{N})$ too. The
semigroup $\mathscr{I}_{\infty}^{\,\Rsh\!\!\!\nearrow}(\mathbb{N})$
is called \emph{the semigroup of co-finite almost monotone partial
bijections} of $\mathbb{N}$.

In the paper \cite{Chuchman-Gutik-2010} the semigroup
$\mathscr{I}_{\infty}^{\,\Rsh\!\!\!\nearrow}(\mathbb{N})$ is studied. It was shown that the semigroup
$\mathscr{I}_{\infty}^{\,\Rsh\!\!\!\nearrow}(\mathbb{N})$ has
algebraic properties similar to the bicyclic semigroup: it is
bisimple and all of its non-trivial group homomorphisms are either
isomorphisms or group homomorphisms. Also it was proved that every Baire shift-continuous
$T_1$-topology $\tau$ on
$\mathscr{I}_{\infty}^{\,\Rsh\!\!\!\nearrow}(\mathbb{N})$  is discrete, described the closure of
$(\mathscr{I}_{\infty}^{\,\Rsh\!\!\!\nearrow}(\mathbb{N}),\tau)$ in
a topological semigroup and constructed non-discrete Hausdorff
semigroup topologies on
$\mathscr{I}_{\infty}^{\,\Rsh\!\!\!\nearrow}(\mathbb{N})$.

A partial transformation $\alpha\colon (X,d)\rightharpoonup (X,d)$ of a metric space $(X,d)$ is called \emph{isometric} or a \emph{partial isometry}, if $d(x\alpha,y\alpha)=d(x,y)$ for all $x,y\in \operatorname{dom}\alpha$. It is obvious that the composition of two partial isometries of a metric space $(X,d)$ is a partial isometry, and the converse partial map to a partial isometry is a partial isometry. Hence the set of partial isometries of a metric space $(X,d)$ with the operation of composition of partial isometries is an inverse submonoid of the symmetric inverse monoid over the set $X$.

Let $\mathbf{I}\mathbb{N}_{\infty}$ be the set of all partial cofinite isometries of the set of positive integers $\mathbb{N}$ with the usual metric $d(n,m)=|n-m|$, $n,m\in \mathbb{N}$. Then $\mathbf{I}\mathbb{N}_{\infty}$ with the operation of composition of partial isometries is an inverse submonoid of $\mathscr{I}_\omega$. The semigroup $\mathbf{I}\mathbb{N}_{\infty}$ of all partial co-finite isometries of positive integers is studied in \cite{Gutik-Savchuk-2018}. There we describe the Green relations on the semigroup $\mathbf{I}\mathbb{N}_{\infty}$, its band and proved that $\mathbf{I}\mathbb{N}_{\infty}$ is a simple $E$-unitary $F$-inverse semigroup. Also in \cite{Gutik-Savchuk-2018}, the least group congruence $\mathfrak{C}_{\mathbf{mg}}$ on $\mathbf{I}\mathbb{N}_{\infty}$ is described and proved that the quotient-semigroup  $\mathbf{I}\mathbb{N}_{\infty}/\mathfrak{C}_{\mathbf{mg}}$ is isomorphic to the additive group of integers $\mathbb{Z}(+)$. An example of a non-group congruence on the semigroup $\mathbf{I}\mathbb{N}_{\infty}$ is presented. Also we proved that a congruence on the semigroup $\mathbf{I}\mathbb{N}_{\infty}$ is group if and only if its restriction onto an isomorphic  copy of the bicyclic semigroup in $\mathbf{I}\mathbb{N}_{\infty}$ is a group congruence.

In this paper we show that every automorphism of a full inverse subsemigroup of $\mathscr{I}_{\infty}^{\!\nearrow}(\mathbb{N})$ which contains the semigroup $\mathscr{C}_{\mathbb{N}}$ is the identity map. We construct a submonoid $\mathbf{I}\mathbb{N}_{\infty}^{[\underline{1}]}$ of $\mathscr{I}_{\infty}^{\,\Rsh\!\!\!\nearrow}(\mathbb{N})$ with the following property: if $S$ be an inverse subsemigroup of $\mathscr{I}_{\infty}^{\,\Rsh\!\!\!\nearrow}(\mathbb{N})$ such that $S$ contains $\mathbf{I}\mathbb{N}_{\infty}^{[\underline{1}]}$ as a submonoid, then every non-identity congruence $\mathfrak{C}$ on $S$ is a group congruence. We show that if $S$ is an inverse submonoid of $\mathscr{I}_{\infty}^{\,\Rsh\!\!\!\nearrow}(\mathbb{N})$ such that $S$ contains $\mathscr{C}_{\mathbb{N}}$ as a subsubmonoid then $S$ is simple and the quotient semigroup $S/\mathfrak{C}_{\mathbf{mg}}$, where $\mathfrak{C}_{\mathbf{mg}}$ is minimum group congruence on $S$, is isomorphic to the additive group of integers. Also, we study topologizations of inverse submonoids of $\mathscr{I}_{\infty}^{\,\Rsh\!\!\!\nearrow}(\mathbb{N})$ which contain $\mathscr{C}_{\mathbb{N}}$ and embeddings of  such semigroups into compact-like topological semigroups.

\section{Main algebraic results}\label{section-2}

We recall for a semigroup $S$ a homomorphism $\Phi\colon S\to S$ is called an \emph{endomorphism of} $S$ and every bijective endomorphism (isomorphism) $\Phi\colon S\to S$ is called an \emph{automorphism of} $S$. We observe that in the case when $S$ is a monoid with the unit $1_S$, then an endomorphism $\Phi\colon S\to S$ with $(1_S)\Phi=1_S$ is called a \emph{monoid endomorphism}. It is obvious that $(1_S)\Phi=1_S$ for any automorphism $\Phi\colon S\to S$ of a monoid with the unit $1_S$.

Recall \cite{Munn-1983} a semigroup $S$ is combinatorial if it has no non-trivial subgroups. A regular (an inverse) semigroup $S$ is combinatorial if all its $\mathscr{H}$-classes are singleton. It is obvious that any subsemigroup of a combinatorial semigroup is combinatorial.

\begin{lemma}\label{lemma-2.1-2}
Let $\Psi\colon S\to S$ be an automorphism of a combinatorial inverse semigroup $S$. If $(e)\Psi=e$ for all $e\in E(S)$ then $\Psi$ is the identity map.
\end{lemma}

\begin{proof}
Fix an arbitrary $s\in S\setminus E(S)$. Then $(ss^{-1})\Psi=ss^{-1}$ and $(s^{-1}s)\Psi=s^{-1}s$. Since in any inverse semigroup the following condition hold: $x\mathscr{H}y$ if and only if $xx^{-1}=yy^{-1}$ and $x^{-1}x=y^{-1}y$ (see \cite[Section 3.2, p.~82]{Lawson-1998}), we have that \begin{equation*}
(s)\Psi(s^{-1})\Psi=(ss^{-1})\Psi=ss^{-1} \qquad \hbox{and} \qquad (s^{-1})\Psi(s)\Psi=(s^{-1}s)\Psi=s^{-1}s,
\end{equation*}
and hence $(s)\Psi\mathscr{H}s$. Since $S$ is a combinatorial inverse semigroup, $(s)\Psi=s$.
\end{proof}

For any positive integer $i$ by $\varepsilon(i)$ we denote the identity map of the set $\mathbb{N}\setminus\{i\}$. It is obvious that $\varepsilon(i)\in E(\mathbf{I}\mathbb{N}_{\infty})$ for any positive integer $i$.

\begin{lemma}\label{lemma-2.1-3}
Let $S$ be a full inverse submonoid of $\mathscr{I}_{\infty}^{\!\nearrow}(\mathbb{N})$ and $\Phi\colon S\to S$ be an automorphism. Then $(\varepsilon(1))\Phi=\varepsilon(1)$.
\end{lemma}

\begin{proof}
Since $\Phi\colon S\to S$ is an automorphism, $(\mathbb{I})\Phi=\mathbb{I}$. Suppose to the contrary that $(\varepsilon(1))\Phi\neq\varepsilon(1)$. Since the restriction $\Phi|_{E(S)\setminus\{\mathbb{I}\}}\colon E(S)\setminus\{\mathbb{I}\}\to E(S)\setminus\{\mathbb{I}\}$ of the automorphism $\Phi$ onto $E(S)\setminus\{\mathbb{I}\}$ is an automorphism, there exist (not necessary distinct) idempotents $\iota,\upsilon\in S\setminus \{\mathbb{I},\varepsilon(1)\}$ such that $(\varepsilon(1))\Phi=\upsilon$, $(\iota)\Phi=\varepsilon(1)$ and $|\mathbb{N}\setminus\operatorname{dom}\upsilon|= |\mathbb{N}\setminus\operatorname{dom}\iota|=1$.

We shall show that $1\in \operatorname{dom}\phi\cap\operatorname{ran}\phi$ and moreover $(1)\phi=1$ for any $\phi\in \langle(\alpha)\Phi,(\beta)\Phi\rangle$. Our assumption implies that $\varepsilon(1))=\beta\alpha$ and hence \begin{equation*}
(1)(\beta\alpha)\Phi=1=(1)(\alpha\beta)\Psi=(1)(\mathbb{I})\Phi=(1)\mathbb{I}=1.
\end{equation*}
This implies that $1\in \operatorname{dom}(\alpha)\Phi$ and $1\in \operatorname{dom}(\beta)\Phi$. If $(1)(\beta)\Phi\neq 1$ then the monotonicity of $\beta$ implies that $1\notin \operatorname{dom}(\alpha)\Phi$, and hence $1\notin \operatorname{dom}(\alpha\beta)\Phi=\mathbb{N}$, a contradiction. Since $\alpha$ is inverse of $\beta$ in $S$, the equality $(1)(\beta)\Phi=1$ implies that
\begin{equation*}
  1=(1)(\beta\alpha)\Phi=((1)(\beta)\Phi)(\alpha)\Psi=(1)(\alpha)\Phi.
\end{equation*}
This implies that $(1)(\beta^i\alpha^j)\Phi=1$ for all non-negative integers $i$ and $j$.

By Remark~\ref{remark-1.1}, $\langle\alpha,\beta\rangle$ is a submonoid of $\mathscr{I}_{\infty}^{\!\nearrow}(\mathbb{N})$ which is isomorphic to the bicyclic monoid, and since $\Phi\colon S\to S$ is an automorphism, $\langle(\alpha)\Phi,(\beta)\Phi\rangle$ is isomorphic to the bicyclic monoid, too. By Lemma~2.6 of \cite{Gutik-Repovs-2011} for every idempotent $\varepsilon\in \mathscr{I}_{\infty}^{\!\nearrow}(\mathbb{N})$ there exists a positive integer $n_\varepsilon$ such that $\varepsilon\cdot \beta^n\alpha^n=\beta^n\alpha^n$ for any positive integer $n\geqslant n_\varepsilon$. Then there exists a positive integer $n_\iota$ such that $\iota\beta^n\alpha^n=\beta^n\alpha^n$ and hence  $(\iota\beta^n\alpha^n)\Phi=(\beta^n\alpha^n)\Phi$ for all $n\geqslant n_\iota$. Since $(\iota)\Phi=\beta\alpha$ we have that $(\iota\beta^n\alpha^n)\Phi=(\iota)\Phi(\beta^n\alpha^n)\Phi=\varepsilon(1)(\beta^n\alpha^n)\Phi$ and hence $1\notin\operatorname{dom}\beta\alpha$ for all $n\geqslant n_\iota$. This contradicts the previous part of the proof. The obtained contradiction implies the statement of the lemma.
\end{proof}

\begin{lemma}\label{lemma-2.1-4}
Let $S$ be a full inverse submonoid of $\mathscr{I}_{\infty}^{\!\nearrow}(\mathbb{N})$ and $\Phi\colon S\to S$ be an automorphism. Then $(\beta^i\alpha^j)\Phi=\beta^i\alpha^j$ for all non-negative integers $i$ and $j$.
\end{lemma}

\begin{proof}
By Lemma~\ref{lemma-2.1-3}, $(\beta\alpha)\Phi=(\varepsilon(1))\Phi=\varepsilon(1)=\beta\alpha$ and since $(\mathbb{I})\Phi=\mathbb{I}$, we have that
\begin{equation*}
  (\beta)\Phi(\alpha)\Phi=\beta\alpha \qquad \hbox{and} \qquad (\alpha)\Phi(\beta)\Phi=\mathbb{I}.
\end{equation*}
By Proposition~2.1$(iii)$ from \cite{Gutik-Repovs-2011} the semigroup $\mathscr{I}_{\infty}^{\!\nearrow}(\mathbb{N})$ is combinatorial and hence $S$ is combinatorial, too. Then the arguments  presented in the proof of Lemma~\ref{lemma-2.1-2} imply that $(\beta)\Phi=\beta$ and $(\alpha)\Phi=\alpha$. Therefore we get
\begin{equation*}
(\beta^i\alpha^j)\Phi=(\beta^i)\Phi(\alpha^j)\Phi=((\beta)\Phi)^i((\alpha)\Phi)^j=\beta^i\alpha^j
\end{equation*}
for all non-negative integers $i$ and $j$.
\end{proof}

\begin{lemma}\label{lemma-2.1-5}
Let $S$ be a full inverse submonoid of $\mathscr{I}_{\infty}^{\!\nearrow}(\mathbb{N})$ and $\Phi\colon S\to S$ be an automorphism. Then $(\varepsilon)\Phi=\varepsilon$ for each idempotent $\varepsilon\in S$.
\end{lemma}

\begin{proof}
Since the restriction $\Phi|_{E(S)\setminus\{\mathbb{I}\}}\colon E(S)\setminus\{\mathbb{I}\}\to E(S)\setminus\{\mathbb{I}\}$ of $\Phi$ onto $E(S)\setminus\{\mathbb{I}\}$ is an automorphism, the equality $(\iota)\Phi=\upsilon$ for $\iota,\upsilon\in E(S)\setminus \{\mathbb{I},\varepsilon(1)\}$ implies that $|\mathbb{N}\setminus\operatorname{dom}\upsilon|= |\mathbb{N}\setminus\operatorname{dom}\iota|$. Fix so elements $\iota,\upsilon\in E(S)\setminus \{\mathbb{I},\varepsilon(1)\}$ with $|\mathbb{N}\setminus\operatorname{dom}\upsilon|= |\mathbb{N}\setminus\operatorname{dom}\iota|=1$. Then there exist positive integers $k$ and $l$ such that $\upsilon=\varepsilon(k)$ and $\iota=\varepsilon(l)$. Suppose to the contrary that $\iota\neq \upsilon$. If $k>l>1$ then
\begin{equation*}
  \beta^l\alpha^l=(\beta^l\alpha^l)\Phi=(\beta^l\alpha^l\cdot \varepsilon(l))\Phi=\beta^l\alpha^l\cdot (\varepsilon(l))\Phi=\beta^l\alpha^l\cdot (\varepsilon(l))\Phi= \beta^l\alpha^l\cdot \varepsilon(k)\neq \beta^l\alpha^l.
\end{equation*}
 If $l>k>1$ then
\begin{equation*}
\begin{split}
  \beta^k\alpha^k & =(\beta^k\alpha^k)\Phi^{-1}=(\beta^k\alpha^k\cdot \varepsilon(k))\Phi^{-1}=\beta^k\alpha^k\cdot (\varepsilon(k))\Phi^{-1}=\beta^k\alpha^k\cdot (\varepsilon(k))\Phi^{-1}= \\
    & =\beta^k\alpha^k\cdot \varepsilon(l)\neq \beta^k\alpha^k.
\end{split}
\end{equation*}
The obtained contradictions and Lemma~\ref{lemma-2.1-4} imply that $(\iota)\Phi=\iota$ for every $\iota\in E(S)$ with $|\mathbb{N}\setminus\operatorname{dom}\iota|=1$.

By Proposition~2.1 of \cite{Gutik-Repovs-2011} for every idempotent $\varepsilon\in \mathscr{I}_{\infty}^{\!\nearrow}(\mathbb{N})$ there exists a finite subset $\{n_1,\ldots,n_k\}$ of positive integers such that $\varepsilon$ is the identity map of $\mathbb{N}\setminus \{n_1,\ldots,n_k\}$. This implies that $\varepsilon=\varepsilon(n_1)\cdots\varepsilon(n_k)$. Hence we get that
\begin{equation*}
  (\varepsilon)\Phi=(\varepsilon(n_1)\cdots\varepsilon(n_k))\Phi= (\varepsilon(n_1))\Phi\cdots(\varepsilon(n_k))\Phi= \varepsilon(n_1)\cdots\varepsilon(n_k)=\varepsilon,
\end{equation*}
which completes the proof of the lemma.
\end{proof}

It is well known that every automorphism $\Phi$ of the bicyclic semigroup $\mathscr{C}(p,q)$ is trivial. i.e., $\Phi$ is the identity map of $\mathscr{C}(p,q)$. The following theorem shows that every full inverse subsemigroup of $\mathscr{I}_{\infty}^{\!\nearrow}(\mathbb{N})$ which contains the semigroup $\mathscr{C}_{\mathbb{N}}$ has such property.

\begin{theorem}\label{theorem-2.1-6}
Let $S$ be a full inverse submonoid of $\mathscr{I}_{\infty}^{\!\nearrow}(\mathbb{N})$ which contains the semigroup $\mathscr{C}_{\mathbb{N}}$. Then every automorphism of $S$ is the identity map.
\end{theorem}

\begin{proof}
By Lemma~\ref{lemma-2.1-5} for each automorphism $\Phi\colon S\to S$ the band $E(\mathscr{I}_{\infty}^{\!\nearrow}(\mathbb{N}))$ is the set of fixed points of $\Phi$. By Proposition~2.1 of \cite{Gutik-Repovs-2011}, $\mathscr{I}_{\infty}^{\!\nearrow}(\mathbb{N})$ is combinatorial inverse semigroup, and hence by Proposition~3.2.11 of \cite{Lawson-1998} so is $S$. Next we apply Lemma~\ref{lemma-2.1-2}.
\end{proof}

Theorem~\ref{theorem-2.1-6} implies the following two corollaries.

\begin{corollary}\label{corollary-2.1-7}
Every automorphism  of the  semigroup $\mathscr{I}_{\infty}^{\!\nearrow}(\mathbb{N})$ is trivial.
\end{corollary}

\begin{corollary}\label{corollary-2.1-8}
Every automorphism  of the  semigroup $\mathbf{I}\mathbb{N}_{\infty}$ is trivial.
\end{corollary}

\begin{remark}\label{remark-2.1-9}
By Lemma~1.1 from \cite{Chuchman-Gutik-2010} the band of the monoid $\mathscr{I}_{\infty}^{\!\nearrow}(\mathbb{N})$ is isomorphic to the free semilattice $(\mathscr{P}_{<\omega}(\omega),\cup)$. Next we identify $\mathbb{N}$ with $\omega$. Then every bijective transformation of $\mathbb{N}$ extends to an automorphism of the free semilattice $(\mathscr{P}_{<\omega}(\omega),\cup)$. This implies that the monoid $\mathscr{I}_{\infty}^{\!\nearrow}(\mathbb{N})$ contains a full inverse subsemigroup which has $\mathfrak{c}$ distinct automorphisms.
\end{remark}

An example of a non-group congruence on the semigroup $\mathbf{I}\mathbb{N}_{\infty}$ is presented in \cite{Gutik-Savchuk-2018}. Later we shall establish what submonoids of $\mathscr{I}_{\infty}^{\,\Rsh\!\!\!\nearrow}(\mathbb{N})$ admit only a group non-identity congruence.

\medskip

For an arbitrary positive integer $n_0$ we denote $[n_0)=\left\{n\in\mathbb{N}\colon n\geqslant n_0\right\}$. Since the set of all positive integers is well ordered, the definition of the semigroup $\mathscr{I}_{\infty}^{\,\Rsh\!\!\!\nearrow}(\mathbb{N})$ implies that for every $\alpha\in\mathscr{I}_{\infty}^{\,\Rsh\!\!\!\nearrow}(\mathbb{N})$ there exists the smallest positive integer $n_{\alpha}^{\mathbf{d}}\in\operatorname{dom}\alpha$ such that the restriction $\alpha|_{\left[n_{\alpha}^{\mathbf{d}}\right)}$ of the partial map $\alpha\colon \mathbb{N}\rightharpoonup \mathbb{N}$ onto the set $\left[n_{\alpha}^{\mathbf{d}}\right)$ is an element of the semigroup $\mathscr{C}_{\mathbb{N}}$, i.e., $\alpha|_{\left[n_{\alpha}^{\mathbf{d}}\right)}$ is a some partial shift of $\left[n_{\alpha}^{\mathbf{d}}\right)$. For every $\alpha\in\mathscr{I}_{\infty}^{\,\Rsh\!\!\!\nearrow}(\mathbb{N})$ we put $\overrightarrow{\alpha}=\alpha|_{\left[n_{\alpha}^{\mathbf{d}}\right)}$, i.e.
\begin{equation*}
\operatorname{dom}\overrightarrow{\alpha}=\left[n_{\alpha}^{\mathbf{d}}\right), \qquad (x)\overrightarrow{\alpha}=(x)\alpha \quad \hbox{for all} \quad x\in \operatorname{dom}\overrightarrow{\alpha} \qquad \hbox{and} \qquad \operatorname{ran}\overrightarrow{\alpha}=\left(\operatorname{dom}\overrightarrow{\alpha}\right)\alpha.
\end{equation*}
Also, we put
\begin{equation*}
\underline{n}_{\alpha}^{\mathbf{d}}=\min\left\{j\in\mathbb{N}\colon j\in\operatorname{dom}\alpha\right\} \qquad \hbox{for} \quad \alpha\in\mathscr{I}_{\infty}^{\,\Rsh\!\!\!\nearrow}(\mathbb{N}),
\end{equation*}
and
\begin{equation*}
\overline{n}_{\alpha}^{\mathbf{d}}=\max\left\{j\in\operatorname{dom}\alpha\colon j<n_{\alpha}^{\mathbf{d}}\right\} \qquad \hbox{for} \quad \alpha\in\mathscr{I}_{\infty}^{\,\Rsh\!\!\!\nearrow}(\mathbb{N})\setminus\mathscr{C}_{\mathbb{N}}.
\end{equation*}
It is obvious that $\underline{n}_{\alpha}^{\mathbf{d}}\leqslant n_{\alpha}^{\mathbf{d}}$ when $\alpha\in\mathscr{I}_{\infty}^{\,\Rsh\!\!\!\nearrow}(\mathbb{N})$ and $\underline{n}_{\alpha}^{\mathbf{d}}\leqslant \overline{n}_{\alpha}^{\mathbf{d}}< n_{\alpha}^{\mathbf{d}}$ when $\alpha\in\mathscr{I}_{\infty}^{\,\Rsh\!\!\!\nearrow}(\mathbb{N})\setminus\mathscr{C}_{\mathbb{N}}$.

The following theorem is proved in \cite{Gutik-Savchuk-2018}.

\begin{theorem}[{\cite[Theorem~9]{Gutik-Savchuk-2018}}]\label{theorem-2.1}
Let $\mathfrak{C}$ be a congruence on the semigroup $\mathbf{I}\mathbb{N}_{\infty}$. Then the following conditions are equivalent:
\begin{enumerate}
  \item[$(1)$] $\mathfrak{C}$ is a group congruence;

  \item[$(2)$] there exists a subsemigroup $S$ of $\mathbf{I}\mathbb{N}_{\infty}$ which is isomorphic to the bicyclic semigroup and $S$ contains two distinct $\mathfrak{C}$-equivalent elements;

  \item[$(3)$] every subsemigroup of $\mathbf{I}\mathbb{N}_{\infty}$, which is isomorphic to the bicyclic semigroup, has two distinct $\mathfrak{C}$-equivalent elements.
\end{enumerate}
\end{theorem}

The following lemma completes the statements of Theorem~\ref{theorem-2.1}.

\begin{lemma}\label{lemma-2.1-1}
Let $\mathfrak{C}$ be a congruence on the semigroup $\mathbf{I}\mathbb{N}_{\infty}$, $\varepsilon\in E(\mathscr{C}_{\mathbb{N}})$, $\iota\in E(\mathbf{I}\mathbb{N}_{\infty})\setminus E(\mathscr{C}_{\mathbb{N}})$ and $\iota\leqslant\varepsilon$. Then $\varepsilon\mathfrak{C}\iota$ implies that $\mathfrak{C}$ is a group congruence on $\mathbf{I}\mathbb{N}_{\infty}$.
\end{lemma}

\begin{proof}
The assumptions of the lemma imply that $n_{\iota}^{\mathbf{d}}<n_{\varepsilon}^{\mathbf{d}}$. Put $\varepsilon_{n_{\iota}^{\mathbf{d}}+1}\colon \mathbb{N}\rightharpoonup\mathbb{N}$ and $\varepsilon_{n_{\iota}^{\mathbf{d}}}\colon \mathbb{N}\rightharpoonup\mathbb{N}$ are identity maps of the sets $\left[n_{\iota}^{\mathbf{d}}+1\right)$ and $\left[n_{\iota}^{\mathbf{d}}\right)$, respectively. It is obvious that $\varepsilon_{n_{\iota}^{\mathbf{d}}+1}, \varepsilon_{n_{\iota}^{\mathbf{d}}}\in E(\mathscr{C}_{\mathbb{N}})$,
\begin{equation*}
  \varepsilon_{n_{\iota}^{\mathbf{d}}}=\varepsilon_{n_{\iota}^{\mathbf{d}}}\cdot \varepsilon_{n_{\iota}^{\mathbf{d}}+1}= \varepsilon_{n_{\iota}^{\mathbf{d}}}\cdot \iota= \varepsilon_{n_{\iota}^{\mathbf{d}}+1}\cdot\iota \qquad \hbox{and} \qquad \varepsilon_{n_{\iota}^{\mathbf{d}}+1}=\varepsilon_{n_{\iota}^{\mathbf{d}}+1}\cdot\varepsilon,
\end{equation*}
and hence $\varepsilon_{n_{\iota}^{\mathbf{d}}+1}\mathfrak{C}\varepsilon_{n_{\iota}^{\mathbf{d}}}$. Then Theorem~\ref{theorem-2.1} and Corollary~1.32~\cite{Clifford-Preston-1961-1967} imply that $\mathfrak{C}$ is a group congruence on $\mathbf{I}\mathbb{N}_{\infty}$.
\end{proof}

\begin{definition}\label{definition-2.2}
Put
\begin{equation*}
\mathbf{I}\mathbb{N}_{\infty}^{[\underline{1}]}=\left\{\alpha\in\mathscr{I}_{\infty}^{\!\nearrow}(\mathbb{N})\colon \hbox{the restriction~} \alpha|_{\operatorname{dom}\alpha\setminus\{\underline{n}_{\alpha}^{\mathbf{d}}\}} \hbox{~is a partial isometry of~} \mathbb{N}\right\}.
\end{equation*}
\end{definition}

It is obvious that $\mathbf{I}\mathbb{N}_{\infty}^{[\underline{1}]}$ is an inverse submonoid of the inverse monoid $\mathscr{I}_{\infty}^{\!\nearrow}(\mathbb{N})$, $\mathbf{I}\mathbb{N}_{\infty}$ is an inverse submonoid of  $\mathbf{I}\mathbb{N}_{\infty}^{[\underline{1}]}$  and $E(\mathbf{I}\mathbb{N}_{\infty})=E(\mathbf{I}\mathbb{N}_{\infty}^{[\underline{1}]})= E(\mathscr{I}_{\infty}^{\!\nearrow}(\mathbb{N}))=E(\mathscr{I}_{\infty}^{\,\Rsh\!\!\!\nearrow}(\mathbb{N}))$.

\begin{lemma}\label{lemma-2.3}
Let $S$ be an inverse subsemigroup of $\mathscr{I}_{\infty}^{\,\Rsh\!\!\!\nearrow}(\mathbb{N})$ such that $S$ contains $\mathbf{I}\mathbb{N}_{\infty}^{[\underline{1}]}$ as a submonoid. Let $\mathfrak{C}$ be a congruence on $S$ such that two distinct idempotents $\varepsilon$ and $\iota$ of $\mathbf{I}\mathbb{N}_{\infty}^{[\underline{1}]}$ are $\mathfrak{C}$-equivalent. Then $\mathfrak{C}$ is a group congruence on $S$.
\end{lemma}

\begin{proof}
If $\varepsilon$ and $\iota$ are idempotents of the subsemigroup $\mathscr{C}_{\mathbb{N}}$ of $\mathscr{I}_{\infty}^{\,\Rsh\!\!\!\nearrow}(\mathbb{N})$ then the statement of our lemma follows from Theorem~\ref{theorem-2.1}. Hence we assume that at least one of idempotents $\varepsilon$ and $\iota$ does not belong to $\mathscr{C}_{\mathbb{N}}$.

We consider two cases:
\begin{equation*}
  1)~n_{\varepsilon}^{\mathbf{d}}=n_{\iota}^{\mathbf{d}}; \qquad \hbox{and} \qquad 2)~n_{\varepsilon}^{\mathbf{d}}\neq n_{\iota}^{\mathbf{d}}.
\end{equation*}

Suppose case $n_{\varepsilon}^{\mathbf{d}}=n_{\iota}^{\mathbf{d}}$ holds. Since $\varepsilon\neq\iota$ without loss of generality we may assume that there exists a positive integer $n_0<n_{\varepsilon}^{\mathbf{d}}$ such that $n_0\in\operatorname{dom}\varepsilon\setminus\operatorname{dom}\iota$. Then $n_0=n_{\varepsilon}^{\mathbf{d}}-(k+1)$ for some positive integer $k$.

For every positive integer $j<n_{\varepsilon}^{\mathbf{d}}-1$ we define a partial bijection $\alpha_j\colon \mathbb{N}\rightharpoonup\mathbb{N}$ in the following way:
\begin{equation*}
  \operatorname{dom}\alpha_j=\left\{j\right\}\cup\left\{n\in\mathbb{N}\colon n\geqslant n_{\varepsilon}^{\mathbf{d}}\right\}, \qquad \operatorname{ran}\alpha_j=\left\{j+1\right\}\cup\left\{n\in\mathbb{N}\colon n\geqslant n_{\varepsilon}^{\mathbf{d}}\right\}
\end{equation*}
and
\begin{equation*}
  (n)\alpha_j=
  \left\{
    \begin{array}{cl}
      n,   & \hbox{if~} n\geqslant n_{\varepsilon}^{\mathbf{d}};\\
      n+1, & \hbox{if~} n=j.
    \end{array}
  \right.
\end{equation*}
Simple verifications show that
\begin{equation*}
  \varepsilon_{n_{\varepsilon}^{\mathbf{d}}-1}= \alpha_{n_{\varepsilon}^{\mathbf{d}}-2}^{-1}\cdots\alpha_{n_0+1}^{-1}\alpha_{n_0}^{-1}\varepsilon\alpha_{n_0}\alpha_{n_0+1}\cdots\alpha_{n_{\varepsilon}^{\mathbf{d}}-2}
\end{equation*}
and
\begin{equation*}
  \varepsilon_{n_{\varepsilon}^{\mathbf{d}}}= \alpha_{n_0}^{-1}\iota\alpha_{n_0}= \alpha_{n_0+1}^{-1}\alpha_{n_0}^{-1}\iota\alpha_{n_0}\alpha_{n_0+1}=\ldots= \alpha_{n_{\varepsilon}^{\mathbf{d}}-2}^{-1}\cdots\alpha_{n_0+1}^{-1}\alpha_{n_0}^{-1}\iota\alpha_{n_0}\alpha_{n_0+1}\cdots\alpha_{n_{\varepsilon}^{\mathbf{d}}-2}
\end{equation*}
are identity maps of the sets $\left\{n\in\mathbb{N}\colon n\geqslant n_{\varepsilon}^{\mathbf{d}}-1\right\}$ and $\left\{n\in\mathbb{N}\colon n\geqslant n_{\varepsilon}^{\mathbf{d}}\right\}$, respectively,  and hence $\varepsilon_{n_{\varepsilon}^{\mathbf{d}}-1}$ and $\varepsilon_{n_{\varepsilon}^{\mathbf{d}}}$ are distinct $\mathfrak{C}$-equivalent idempotents of the subsemigroup $\mathscr{C}_{\mathbb{N}}$ in $\mathscr{I}_{\infty}^{\,\Rsh\!\!\!\nearrow}(\mathbb{N})$. By Theorem~\ref{theorem-2.1} all idempotents of the sebsemigroup $\mathbf{I}\mathbb{N}_{\infty}$ are $\mathfrak{C}$-equivalent, and hence $\mathfrak{C}$ is a group congruence on the semigroup $S$, because $E(\mathbf{I}\mathbb{N}_{\infty})= E(S)=E(\mathscr{I}_{\infty}^{\,\Rsh\!\!\!\nearrow}(\mathbb{N}))$.

Suppose case $n_{\varepsilon}^{\mathbf{d}}\neq n_{\iota}^{\mathbf{d}}$ holds. Without loss of generality we may assume that $n_{\varepsilon}^{\mathbf{d}}> n_{\iota}^{\mathbf{d}}$. Put $\varepsilon_{n_{\iota}^{\mathbf{d}}-1}\colon \mathbb{N}\rightharpoonup\mathbb{N}$ is the identity map of the set $\left\{n\in\mathbb{N}\colon n\geqslant n_{\varepsilon}^{\mathbf{d}}-1\right\}$. Simple verifications show that $\varepsilon_{n_{\iota}^{\mathbf{d}}-1}=\varepsilon_{n_{\iota}^{\mathbf{d}}-1}\varepsilon$ and $\overrightarrow{\iota}=\varepsilon_{n_{\iota}^{\mathbf{d}}-1}\iota$ are distinct $\mathfrak{C}$-equivalent idempotents of the subsemigroup $\mathscr{C}_{\mathbb{N}}$ in $\mathscr{I}_{\infty}^{\,\Rsh\!\!\!\nearrow}(\mathbb{N})$. By Theorem~\ref{theorem-2.1} all idempotents of the sebsemigroup $\mathbf{I}\mathbb{N}_{\infty}$ are $\mathfrak{C}$-equivalent, and hence $\mathfrak{C}$ is a group congruence on the semigroup $S$, because $E(\mathbf{I}\mathbb{N}_{\infty})= E(S)=E(\mathscr{I}_{\infty}^{\,\Rsh\!\!\!\nearrow}(\mathbb{N}))$.
\end{proof}

\begin{theorem}\label{theorem-2.4}
Let $S$ be an inverse subsemigroup of $\mathscr{I}_{\infty}^{\,\Rsh\!\!\!\nearrow}(\mathbb{N})$ such that $S$ contains $\mathbf{I}\mathbb{N}_{\infty}^{[\underline{1}]}$ as a submonoid. Then every non-identity congruence $\mathfrak{C}$ on $S$ is a group congruence.
\end{theorem}

\begin{proof}
Let $\alpha$ and $\beta$ be two distinct $\mathfrak{C}$-equivalent elements of the semigroup $S$.

We consider two cases:
\begin{itemize}
  \item[$(i)$] $\alpha\mathscr{H}\beta$ in $S$;
  \item[$(ii)$] $\alpha$ and $\beta$ belong to distinct two $\mathscr{H}$-classes in $S$.
\end{itemize}

Suppose that $\alpha\mathscr{H}\beta$ in $S$. Then Proposition~1.1$(ix)$ of \cite{Chuchman-Gutik-2010} and Proposition~3.2.11 of \cite{Lawson-1998} imply that $\operatorname{dom}\alpha=\operatorname{dom}\beta$ and $\operatorname{ran}\alpha=\operatorname{ran}\beta$, and hence there exists a positive integer $n_0\in \operatorname{dom}\alpha$ such that $(n_0)\alpha\neq (n_0)\beta$. Let $\varepsilon_{n_0}\colon\mathbb{N}\rightharpoonup\mathbb{N}$ be the identity map of the set $\left\{n_0\right\}\cup\left\{n\in\mathbb{N}\colon n\geqslant m_0\right\}$, where $m_0\in\operatorname{dom}\alpha$ is an arbitrary positive integer such that $m_0\geqslant n_0+n_{\alpha}^{\mathbf{d}}$. By Proposition~3$(i)$ of \cite{Gutik-Savchuk-2018} and Proposition~3$(i)$ of \cite{Chuchman-Gutik-2010}, $E(\mathbf{I}\mathbb{N}_{\infty})=E(\mathscr{I}_{\infty}^{\,\Rsh\!\!\!\nearrow}(\mathbb{N}))$ and hence $\varepsilon_{n_0}\in E(S)$. Since $S$ is an inverse semigroup Proposition~2.3.4 from \cite{Lawson-1998} and $\alpha\mathfrak{C}\beta$ imply that $\alpha^{-1}\mathfrak{C}\beta^{-1}$, and hence we have that $(\alpha^{-1}\varepsilon_{n_0}\alpha)\mathfrak{C}(\beta^{-1}\varepsilon_{n_0}\beta)$. Then the definition of $\varepsilon_{n_0}$ implies that $\alpha^{-1}\varepsilon_{n_0}\alpha$ and $\beta^{-1}\varepsilon_{n_0}\beta$ are distinct idempotents of the semigroup $S$, and hence by Lemma~\ref{lemma-2.3}, $\mathfrak{C}$ is a group congruence on $S$.

If case $(ii)$ holds then at least one of the following conditions holds
\begin{equation*}
  \alpha\alpha^{-1}\neq\beta\beta^{-1} \qquad \hbox{or} \qquad \alpha^{-1}\alpha\neq\beta^{-1}\beta.
\end{equation*}
Then by Proposition~2.3.4 of \cite{Lawson-1998} the semigroup $S$ has two distinct $\mathfrak{C}$-equivalent idempotents. Next we apply Lemma~\ref{lemma-2.3}.
\end{proof}

Every inverse semigroup $S$ admits the \emph{least group congruence} $\mathfrak{C}_{\mathbf{mg}}$ (see \cite[Section III]{Petrich-1984}):
\begin{equation*}
s\mathfrak{C}_{\mathbf{mg}}t \qquad \hbox{if and only if there exists an idempotent} \quad e\in S \quad \hbox{such that} \quad se=te.
\end{equation*}

Later we shall describe the least group congruence on any inverse subsemigroup $S$ of $\mathscr{I}_{\infty}^{\,\Rsh\!\!\!\nearrow}(\mathbb{N})$ such that $S$ contains $\mathscr{C}_{\mathbb{N}}$ as a submonoid.

\smallskip

Definitions of inverse semigroups $\mathscr{C}_{\mathbb{N}}$, $\mathscr{I}_{\infty}^{\,\Rsh\!\!\!\nearrow}(\mathbb{N})$ and the congruence $\mathfrak{C}_{\mathbf{mg}}$ imply the following lemma.

\begin{lemma}\label{lemma-2.5}
Let $S$ be an inverse subsemigroup of $\mathscr{I}_{\infty}^{\,\Rsh\!\!\!\nearrow}(\mathbb{N})$ such that $S$ contains $\mathscr{C}_{\mathbb{N}}$ as a submonoid. Then the following conditions hold:
\begin{enumerate}
  \item[$(i)$] $\alpha\mathfrak{C}_{\mathbf{mg}}\overrightarrow{\alpha}$ for every $\alpha\in S$;
  \item[$(ii)$] if $\alpha$ and $\beta$ are elements of $S$ such that $\alpha=\overrightarrow{\alpha}$ and $\beta=\overrightarrow{\beta}$, then $\alpha\mathfrak{C}_{\mathbf{mg}}\beta$ if and only if $(n)\alpha=(n)\beta$ for all $n\in\operatorname{dom}\alpha\cap\operatorname{dom}\beta$.
\end{enumerate}
\end{lemma}

\begin{theorem}\label{theorem-2.6}
Let $S$ be an inverse subsemigroup of $\mathscr{I}_{\infty}^{\,\Rsh\!\!\!\nearrow}(\mathbb{N})$ such that $S$ contains $\mathscr{C}_{\mathbb{N}}$ as a submonoid. Then the quotient semigroup $S/\mathfrak{C}_{\mathbf{mg}}$ is isomorphic to the additive group of integers $\mathbb{Z}(+)$.
\end{theorem}

\begin{proof}
We define a map $\mathfrak{F}\colon S\to\mathbb{Z}(+)$, $\alpha\mapsto \mathfrak{i}_{\alpha}$  in the following way. Put $\mathfrak{i}_{\alpha}=(n)\overrightarrow{\alpha}-n$, where $n\in \operatorname{dom}\overrightarrow{\alpha}$. Simple verification implies that so defined map $\mathfrak{F}$ is correct and it is a homomorphism. Also, Lemma~\ref{lemma-2.5} implies that $\alpha\mathfrak{C}_{\mathbf{mg}}\beta$ if and only if $(\alpha)\mathfrak{F}=(\beta)\mathfrak{F}$ for $\alpha,\beta\in S$.
\end{proof}

Theorems~\ref{theorem-2.4} and~\ref{theorem-2.6} imply the following corollary.

\begin{corollary}\label{corollary-2.7}
Let $S$ be an inverse subsemigroup of $\mathscr{I}_{\infty}^{\,\Rsh\!\!\!\nearrow}(\mathbb{N})$ such that $S$ contains $\mathbf{I}\mathbb{N}_{\infty}^{[\underline{1}]}$ as a submonoid. Then for any non-injective homomorphism $\mathfrak{F}\colon S\to T$ into an arbitrary semigroup $T$ there exists a unique homomorphism $\mathfrak{H}\colon\mathbb{Z}(+)\to T$ such that the following diagram
\begin{equation*}
\xymatrix{
S\ar[r]^{\mathfrak{F}}\ar[d]_{\mathfrak{C}_{\mathbf{mg}}^{\sharp}} & T\\
\mathbb{Z}(+)\ar[ru]_{\mathfrak{H}}
}
\end{equation*}
commutes.
\end{corollary}

The semigroups $\mathscr{C}_{\mathbb{N}}$, $\mathscr{I}_{\infty}^{\!\nearrow}(\mathbb{N})$ and $\mathscr{I}_{\infty}^{\,\Rsh\!\!\!\nearrow}(\mathbb{N})$ are bisimple (see \cite{Clifford-Preston-1961-1967}, \cite{Gutik-Repovs-2011}, \cite{Chuchman-Gutik-2010}). But the semigroup $\mathbf{I}\mathbb{N}_{\infty}$ is not bisimple whereas it is simple. A very amazing property about some inverse subsemigroups of $\mathscr{I}_{\infty}^{\,\Rsh\!\!\!\nearrow}(\mathbb{N})$ illustrates the following theorem.

\begin{theorem}\label{theorem-2.8}
Let $S$ be an inverse subsemigroup of $\mathscr{I}_{\infty}^{\,\Rsh\!\!\!\nearrow}(\mathbb{N})$ such that $S$ contains $\mathscr{C}_{\mathbb{N}}$ as a submonoid. Then $S$ is simple.
\end{theorem}

\begin{proof}
Since $\alpha=\alpha\mathbb{I}=\mathbb{I}\alpha$ for any element $\alpha$ of $S$, it is sufficient to show that for every $\beta\in S$ there exist $\gamma,\delta\in S$ such that $\gamma\beta\delta=\mathbb{I}$.

Fix an arbitrary element $\beta$ in $S$. Simple verifications show that $\beta\overrightarrow{\beta}^{-1}=\overrightarrow{\beta}\overrightarrow{\beta}^{-1}$ and $\beta\overrightarrow{\beta}^{-1}$ is an idempotent of $S$, where $\overrightarrow{\beta}^{-1}$ is inverse of $\overrightarrow{\beta}$ in $S$, because $\overrightarrow{\beta}$ and $\overrightarrow{\beta}^{-1}$ are elements of the sebsemigroup  $\mathscr{C}_{\mathbb{N}}$ in $S$. Next we define a partial maps $\gamma\colon\mathbb{N}\rightharpoonup \mathbb{N}$ in the following way
\begin{equation*}
  \operatorname{dom}\gamma=\mathbb{N}, \qquad \operatorname{ran}\gamma=\left\{n\in\mathbb{N}\colon n\geqslant n_{\gamma}^{\mathbf{d}}\right\} \qquad \hbox{and} \qquad (i)\gamma=i-1+n_{\gamma}^{\mathbf{d}}  \quad \hbox{for} \quad i\in\operatorname{dom}\gamma.
\end{equation*}
Then $\gamma\beta(\overrightarrow{\beta}^{-1}\gamma^{-1})=\mathbb{I}$.
\end{proof}


\section{On shift-continuous topologies on inverse subsemigroups of $\mathscr{I}_{\infty}^{\,\Rsh\!\!\!\nearrow}(\mathbb{N})$}

 A subset $A$ of a topological space $X$ is said to be \emph{co-dense} in $X$ if $X\setminus A$ is dense in $X$.

\smallskip

We recall that a topological space X is said to be:
\begin{itemize}
    \item \emph{compact} if every open cover of $X$ contains a finite subcover;
    \item \emph{countably compact} if each closed discrete subspace of $X$ is finite;
    \item \emph{feebly compact} if each locally finite open cover of $X$ is finite;
    \item \emph{pseudocompact} if $X$ is Tychonoff and each continuous real-valued function on $X$ is bounded;
    \item \emph{locally compact} if each point of $X$ has an open neighbourhood with the compact closure;
    \item \emph{\v{C}ech-complete} if $X$ is Tychonof and there exists a compactifcation $cX$ of $X$ such that the remainder $cX\setminus c(X)$ is an $F_\sigma$-set in $cX$;
    \item \emph{a Baire space} if for each sequence $A_1,A_2,\ldots,A_i,\ldots$ of nowhere dense subsets of $X$ the union $\bigcup_{i=1}^\infty A_i$ is a co-dense subset of $X$.
\end{itemize}
According to Theorem~3.10.22 of \cite{Engelking-1989}, a Tychonoff topological space $X$ is feebly compact if and only if $X$ is pseudocompact. Also, a Hausdorff topological space $X$ is feebly compact if and only if every locally finite family of non-empty open subsets of $X$ is finite. Every compact space is countably compact and every countably compact space is feebly compact (see \cite{Arkhangelskii-1992}). Also, every compact space is locally compact, every locally compact space is \v{C}ech-complete, and every \v{C}ech-complete space is a Baire space (see \cite{Engelking-1989}).

By the Eberhart-Selden theorem every Hausdorff semigroup topology on the bicyclic semigroup is discrete. It is natural to ask: \emph{Do there exists non-discrete semigroup topology on the semigroup $\mathbf{I}\mathbb{N}_{\infty}$?}

\begin{theorem}\label{theorem-3.1}
Let $S$ be an inverse subsemigroup of $\mathscr{I}_{\infty}^{\,\Rsh\!\!\!\nearrow}(\mathbb{N})$ such that $S$ contains $\mathscr{C}_{\mathbb{N}}$ as a submonoid. Then every Baire shift-continuous Hausdorff topology $\tau$ on $S$ is discrete.
\end{theorem}

\begin{proof}
If no point in $S$ is isolated,  then since the space $(S,\tau)$ is Hausdorff, it follows that $\{\alpha\}$ is nowhere dense for all $\alpha\in S$. But, if this is the case, then since the semigroup $S$ is countable it cannot be a Baire space. Hence the space $(S,\tau)$  contains an isolated point $\mu$. If $\gamma\in S$ is arbitrary, then by Theorem~\ref{theorem-2.8}, there exist $\alpha,\beta\in S$ such that $\alpha\cdot\gamma\cdot\beta=\mu$. The map $f\colon\chi\mapsto\alpha\cdot\chi\cdot\beta$ is continuous and so the full preimage $(\{\mu\})f^{-1}$ is open. By Proposition~1.2 from \cite{Chuchman-Gutik-2010} for every
$\alpha,\beta\in\mathscr{I}_{\infty}^{\,\Rsh\!\!\!\nearrow}(\mathbb{N})$,
both sets
 $
\{\chi\in\mathscr{I}_{\infty}^{\,\Rsh\!\!\!\nearrow}(\mathbb{N})\mid
\alpha\cdot\chi=\beta\}
 $
 and
 $
\{\chi\in\mathscr{I}_{\infty}^{\,\Rsh\!\!\!\nearrow}(\mathbb{N})\mid
\chi\cdot\alpha=\beta\}
 $
are finite, and hence the same holds for the subsemigroup $S$ of $\mathscr{I}_{\infty}^{\,\Rsh\!\!\!\nearrow}(\mathbb{N})$. This implies that the set
$(\{\mu\})f^{-1}$ is finite and since $(S,\tau)$ is Hausdorff, $\{\gamma\}$ is open, and hence isolated.
\end{proof}

Since every \v{C}ech complete space (and hence every locally compact
space) is Baire, Theorem~\ref{theorem-3.1} implies
Corollary~\ref{corollary-3.2}.

\begin{corollary}\label{corollary-3.2}
Let $S$ be an inverse subsemigroup of $\mathscr{I}_{\infty}^{\,\Rsh\!\!\!\nearrow}(\mathbb{N})$ such that $S$ contains $\mathscr{C}_{\mathbb{N}}$ as a submonoid. Then every Hausdorff \v{C}ech complete (locally compact) shift-continuous topology $\tau$
on $S$  is discrete.
\end{corollary}

The following example shows that there exists a non-discrete
Tychonoff inverse semigroup topology $\tau_W$ on the semigroup
$\mathbf{I}\mathbb{N}_{\infty}$.

\begin{example}\label{example-3.4}
We define a topology $\tau_{W}$ on the semigroup $\mathbf{I}\mathbb{N}_{\infty}$ as follows. For every $\alpha\in\mathbf{I}\mathbb{N}_{\infty}$ we define a family
\begin{equation*}
    \mathscr{B}_{W}(\alpha)=\left\{U_\alpha(F)\mid F \mbox{ is a finite subset of } \operatorname{dom}\alpha\right\},
\end{equation*}
where
\begin{equation*}
    U_\alpha(F)= \left\{\beta\in\mathbf{I}\mathbb{N}_{\infty} \mid \operatorname{dom}\beta\subseteq\operatorname{dom}\alpha
    \mbox{ and }  (x)\beta=(x)\alpha \mbox{ for all } x\in F\right\}.
\end{equation*}
It is straightforward to verify that
$\{\mathscr{B}_{W}(\alpha)\}_{\alpha\in
\mathscr{I}^{\!\nearrow}_{\infty}(\mathbb{Z})}$ forms a basis for a
topology $\tau_{W}$ on the semigroup
$\mathbf{I}\mathbb{N}_{\infty}$.
\end{example}

\begin{proposition}\label{proposition-3.5}
$(\mathbf{I}\mathbb{N}_{\infty},\tau_{W})$ is a Tychonoff topological inverse semigroup.
\end{proposition}

\begin{proof}
Let $\alpha$ and $\beta$ be arbitrary elements of the semigroup $\mathbf{I}\mathbb{N}_{\infty}$. We put $\gamma=\alpha\cdot\beta$ and let $F=\{n_1,\ldots,n_i\}$ be a finite subset of $\operatorname{dom}\gamma$. We denote $m_1=(n_1)\alpha,\ldots,m_i=(n_i)\alpha$ and $k_1=(n_1)\gamma,\ldots,k_i=(n_i)\gamma$. Then we get that $(m_1)\beta=k_1,\ldots,(m_i)\beta=k_i$. Hence we have that
\begin{equation*}
    U_\alpha(\{n_1,\ldots,n_i\})\cdot
    U_\beta(\{m_1,\ldots,m_i\})\subseteq
    U_\gamma(\{n_1,\ldots,n_i\})
\end{equation*}
and
\begin{equation*}
    \big(U_\gamma(\{n_1,\ldots,n_i\})\big)^{-1}\subseteq
    U_{\gamma^{-1}}(\{k_1,\ldots,k_i\}).
\end{equation*}
Therefore the semigroup operation and the inversion are continuous in $(\mathbf{I}\mathbb{N}_{\infty},\tau_{W})$.

Let $N=\mathbb{N}\cup\{a\}$ for some $a\notin\mathbb{N}$. Then $N^N$  with the operation composition is a semigroup and the map
$\Psi\colon\mathbf{I}\mathbb{N}_{\infty}\rightarrow N^N$ defined by the formula
\begin{equation*}
    (x)(\alpha)\Psi=
\left\{
  \begin{array}{ll}
    (x)\alpha, & \hbox{if } x\in\operatorname{dom}\alpha; \\
    a,         & \hbox{if } x\notin\operatorname{dom}\alpha
  \end{array}
\right.
\end{equation*}
is a monomorphism. Hence $N^N$ is a topological semigroup with the product topology if $N$ has the discrete topology. Obviously, this topology generates topology $\tau_W$ on
$\mathbf{I}\mathbb{N}_{\infty}$. Therefore by Theorem~2.3.11 from \cite{Engelking-1989} topological space $N^N$ is Tychonoff and hence by Theorem~2.1.6 from \cite{Engelking-1989} so is $(\mathbf{I}\mathbb{N}_{\infty},\tau_{W})$. This completes the proof of the proposition.
\end{proof}

\begin{theorem}\label{theorem-3.6}
Let $S$ be an inverse subsemigroup of $\mathscr{I}_{\infty}^{\,\Rsh\!\!\!\nearrow}(\mathbb{N})$ such that $S$ contains $\mathscr{C}_{\mathbb{N}}$ as a submonoid.
Let $T$ be a $T_1$ semitopological semigroup which contains $S$ as a dense discrete subsemigroup. If $I=T\setminus S \neq\varnothing$ then $I$ is an ideal of $T$.
\end{theorem}

\begin{proof}
By Lemma~3 \cite{Gutik-Savchuk-2017}, $S$ is an open subspace of the topological space $T$.

Fix an arbitrary element $y\in I$. If $x\cdot y=z\notin I$ for some $x\in S$ then there exists an open neighbourhood $U(y)$ of the point $y$ in the space $T$ such that $\{x\}\cdot U(y)=\{z\}\subset S$. By Proposition~1.2 from \cite{Chuchman-Gutik-2010} the open neighbourhood $U(y)$ should contain finitely many elements of the semigroup $S$ which contradicts our assumption. Hence $x\cdot y\in I$ for all $x\in S$ and $y\in I$. The proof of the statement that $y\cdot x\in I$ for all $x\in S$ and $y\in I$ is similar.

Suppose to the contrary that $x\cdot y=w\notin I$ for some $x,y\in I$. Then $w\in S$ and the separate continuity of the semigroup operation in $T$ yields open neighbourhoods $U(x)$ and $U(y)$ of the points $x$ and $y$ in the space $T$, respectively, such that $\{x\}\cdot U(y)=\{w\}$ and $U(x)\cdot \{y\}=\{w\}$. Since both neighbourhoods $U(x)$ and $U(y)$ contain infinitely many elements of the semigroup $S$,  equalities $\{x\}\cdot U(y)=\{w\}$ and $U(x)\cdot \{y\}=\{w\}$ do not hold, because $\{x\}\cdot \left(U(y)\cap S\right)\subseteq I$. The obtained contradiction implies that $x\cdot y\in I$.
\end{proof}

Theorem~\ref{theorem-3.6} implies the following corollary:

\begin{corollary}\label{corollary-3.6-1}
Let $T$ be a $T_1$ semitopological semigroup which contains $\mathbf{I}\mathbb{N}_{\infty}$ as a dense discrete submonoid. If $I=T\setminus \mathbf{I}\mathbb{N}_{\infty} \neq\varnothing$ then $I$ is an ideal of $T$.
\end{corollary}

\begin{proposition}\label{proposition-3.7}
Let $S$ be an inverse subsemigroup of $\mathscr{I}_{\infty}^{\,\Rsh\!\!\!\nearrow}(\mathbb{N})$ such that $S$ contains $\mathscr{C}_{\mathbb{N}}$ as a submonoid.
Let $T$ be a Hausdorff topological semigroup which contains $S$ as a dense discrete subsemigroup. Then for every $\gamma\in S$ the set
\begin{equation*}
    D_\gamma=\left\{(\chi,\varsigma)\in S \times S\mid     \chi\cdot \varsigma=\gamma\right\}
\end{equation*}
is a closed-and-open subset of $T\times T$.
\end{proposition}

\begin{proof}
Since $S$ is a discrete subspace of $T$ by Lemma~3 \cite{Gutik-Savchuk-2017} we have that $D_\gamma$ is an open subset of $T\times T$.

Suppose that there exists $\gamma\in S$ such that $D_\gamma$ is a non-closed subset of $T\times T$. Then there exists an accumulation point $(\alpha,\beta)\in T\times T$ of the set $D_\gamma$. The continuity of the semigroup operation in $T$ implies that $\alpha\cdot\beta=\gamma$. But $S\times S$ is a discrete subspace of $T\times T$ and hence by Theorem~\ref{theorem-3.6}, the points $\alpha$ and $\beta$ belong to the ideal $I=T\setminus S$ and hence $\alpha\cdot \beta\in T\setminus S$ cannot be equal to $\gamma$.
\end{proof}

\begin{theorem}\label{theorem-3.8}
Let $S$ be an inverse subsemigroup of $\mathscr{I}_{\infty}^{\,\Rsh\!\!\!\nearrow}(\mathbb{N})$ such that $S$ contains $\mathscr{C}_{\mathbb{N}}$ as a submonoid.
If a $T_1$ topological semigroup $T$ contains $S$ as a dense discrete subsemigroup then the square $T\times T$ cannot be feebly compact.
\end{theorem}

\begin{proof}
By Proposition~\ref{proposition-3.7}, for every $c\in S$ the square $T\times T$ contains an open-and-closed discrete subspace $D_c$.   If we identify the elements of the semigroup $\mathscr{C}_{\mathbb{N}}$  with the elements the bicyclic monoid ${\mathscr{C}}(p,q)$ by an isomorphism $\mathfrak{h}\colon {\mathscr{C}}(p,q)\to \mathscr{C}_{\mathbb{N}}$, then the subspace $D_c$ contains an infinite subset
\begin{equation*}
\left\{\left((q^i)\mathfrak{h},(p^i)\mathfrak{h}\right)\colon i\in\mathbb{N}_0\right\}
\end{equation*}
and hence the set $D_c$ is infinite. This implies that the square $S\times S$ is not feebly compact.
\end{proof}

A topological semigroup $S$ is called
\emph{$\Gamma$-compact} if for every $x\in S$ the closure of the set
$\{x,x^2,x^3,\ldots\}$ is compact in $S$ (see
\cite{Hildebrant-Koch-1986}). The results obtained in
\cite{Anderson-Hunter-Koch-1965}, \cite{Banakh-Dimitrova-Gutik-2009},
\cite{Banakh-Dimitrova-Gutik-2010}, \cite{Gutik-Repovs-2007},
\cite{Hildebrant-Koch-1986}     imply the following

\begin{corollary}\label{corollary-3.9}
Let $S$ be an inverse subsemigroup of $\mathscr{I}_{\infty}^{\,\Rsh\!\!\!\nearrow}(\mathbb{N})$ such that $S$ contains $\mathscr{C}_{\mathbb{N}}$ as a submonoid. If a Hausdorff topological semigroup $T$ satisfies one of the following conditions:
\begin{itemize}
  \item[$(i)$] $T$ is compact;
  \item[$(ii)$] $T$ is $\Gamma$-compact;
  \item[$(iii)$] $T$ is a countably compact topological inverse semigroup;
  \item[$(iv)$] the square $T\times T$ is countably compact;
  \item[$(v)$] the square $T\times T$ is a Tychonoff pseudocompact space,
\end{itemize}
then $T$ does not contain the semigroup $S$ and for every homomorphism $\mathfrak{h}\colon S\to T$ the image $(S)\mathfrak{h}$ is a cyclic subgroup of $T$. Moreover, for every homomorphism $\mathfrak{h}\colon S\to T$ there exists a unique homomorphism $\mathfrak{u}_\mathfrak{h}\colon \mathbb{Z}(+)\to T$ such that the following diagram
\begin{equation*}
\xymatrix{
S\ar[r]^{\mathfrak{h}}\ar[d]_{\mathfrak{C}_{\mathbf{mg}}^{\sharp}} & T\\
\mathbb{Z}(+)\ar[ru]_{\mathfrak{u}_\mathfrak{h}}
}
\end{equation*}
commutes.
\end{corollary}

Recall~\cite{DeLeeuwGlicksberg1961} that a \emph{Bohr
compactification} of a topological semigroup $S$ is a~pair $(\beta,
B(S))$ such that $B(S)$ is a compact topological semigroup,
$\beta\colon S\to B(S)$ is a continuous homomorphism, and if
$g\colon S\to T$ is a continuous homomorphism of $S$ into a compact
semigroup $T$, then there exists a unique continuous homomorphism
$f\colon B(S)\to T$ such that the diagram
\begin{equation*}
\xymatrix{ S\ar[rr]^\beta\ar[dr]_g && B(S)\ar[ld]^f\\
& T &}
\end{equation*}
commutes. Then Corollary~\ref{corollary-3.9} and Proposition~2 from \cite{Anderson-Hunter-1969} imply the
following:

\begin{corollary}\label{corollary-3.16}
Let $S$ be an inverse subsemigroup of $\mathscr{I}_{\infty}^{\,\Rsh\!\!\!\nearrow}(\mathbb{N})$ such that $S$ contains $\mathscr{C}_{\mathbb{N}}$ as a submonoid.
The Bohr compactification of the discrete semigroup $S$  is topologically isomorphic to the Bohr compactification of discrete group $\mathbb{Z}(+)$.
\end{corollary}

\end{document}